\documentclass[11pt]{article}

\usepackage[a4paper,margin=1in]{geometry}
\usepackage[T1]{fontenc}
\usepackage[utf8]{inputenc}
\usepackage{lmodern}
\usepackage{microtype}
\usepackage{amsmath,amssymb,amsthm,mathtools,mathrsfs}
\usepackage{enumitem}
\usepackage{hyperref}
\usepackage{xcolor}
\usepackage{bbm}
\usepackage{tikz}
\usetikzlibrary{arrows.meta,positioning,calc}
\usepackage{comment}
\hypersetup{
  colorlinks=true,
  linkcolor=blue!60!black,
  citecolor=blue!60!black,
  urlcolor=blue!60!black
}

\newtheorem{theorem}{Theorem}[section]
\newtheorem{proposition}[theorem]{Proposition}
\newtheorem{lemma}[theorem]{Lemma}
\newtheorem{corollary}[theorem]{Corollary}
\theoremstyle{remark}
\newtheorem{remark}[theorem]{Remark}

\theoremstyle{plain}
\newtheorem{assumption}[theorem]{Assumption}
\newcommand{\Pbb}{\mathbb{P}}
\newcommand{\Ebb}{\mathbb{E}}
\newcommand{\1}{\mathbbm{1}}
\newcommand{\dd}{\mathrm{d}}

\newcommand{\cL}{\mathcal{L}}

\title{Macroscopic loops in the random loop model on sparse random graphs}
\author{
Andreas Klippel
\thanks{Fachbereich Mathematik, Technische Universit\"at Darmstadt, Schlossgartenstrasse 7, 64289 Darmstadt, Germany}
\\
\texttt{andreas.klippel@tu-darmstadt.de}
}

\begin{document}
\maketitle

\begin{abstract}
We study the random loop model with crosses and bars on sparse random graphs. Our main objective is to prove the existence of macroscopic loops, in the sense that a loop visits a positive proportion of the vertices. We develop a deterministic drift method on arbitrary finite graphs based on three ingredients: a local split--merge--rewire analysis for the loop number, an exact differential identity for the partition function, and a slice estimate reducing the relevant same-loop insertion volume to induced edge counts of small vertex sets. This yields a general criterion in terms of a small-set sparsity condition on the underlying graph. We then verify this condition for random regular graphs, sparse Erd\H{o}s--R\'enyi graphs, and simple bounded-degree configuration models, obtaining averaged lower bounds on the probability of a macroscopic loop whenever the edge density exceeds an explicit threshold depending on the loop weight \(\theta\) and the cross parameter \(u\). For integer values of \(\theta\), a trace representation of the partition function implies log-convexity, which upgrades the averaged bounds to pointwise-in-time results away from the threshold time.
\end{abstract}

\section{Introduction}
Random cycle and loop models form a natural meeting point between probability theory and
quantum statistical mechanics. On the one hand, they arise as random permutation models such
as the interchange process; on the other hand, they appear as graphical representations of
quantum spin systems, going back to the work of T\'oth, Aizenman-Nachtergaele. In these
representations, questions about magnetic order and correlations are translated into geometric
questions about the cycle or loop structure. This point of view has proved very fruitful over
the last decades; see for instance \cite{Toth1993,AizenmanNachtergaele1994,
GoldschmidtUeltschiWindridge2011,Ueltschi2013}.
A central question in this area is whether the model exhibits only small loops or whether long,
in particular macroscopic, loops occur. The answer is known to depend strongly on the geometry
of the underlying graph. In low-dimensional quantum systems, Mermin--Wagner-type phenomena
rule out spontaneous breaking of continuous symmetries, and correspondingly one expects the
absence of macroscopic loop order in dimensions one and two \cite{MerminWagner1966}. This
picture is supported by rigorous decay results for loop correlations in two-dimensional quantum
systems with continuous symmetry \cite{BenassiFrohlichUeltschi2017}. By contrast, in higher
dimensions one expects long loops and long-range order, and this is indeed known in several
cases through reflection positivity and infrared-bound methods in the quantum setting
\cite{Ueltschi2013,bjornberg2022reflection,betz2026improved}.
Beyond Euclidean lattices, long-cycle and long-loop phenomena have been established on a number
of specific geometries. On the complete graph, Schramm proved the emergence of macroscopic
cycles for the interchange process together with Poisson--Dirichlet asymptotics
\cite{Schramm2005}. For interchange-type models with reversals, Bj\"ornberg, Kotowski, Lees,
and Mi{\l}o{\'s} showed on the complete graph that the presence of reversals leads to a
split-versus-twist mechanism and to Poisson--Dirichlet asymptotics with parameter \(1/2\),
thereby extending Schramm's picture beyond the pure interchange case
\cite{BjornbergKotowskiLeesMilos2019}. On trees, Angel introduced the cyclic-time random walk
and its connection to infinite permutations \cite{Angel2003}, and later work of Hammond and of
Betz, Ehlert, Lees, and Roth established sharp phase-transition results for random stirring and
random loop models on trees \cite{Hammond2015,BetzEhlertLeesRoth2021}. For finite
high-dimensional graphs, phase transitions and macroscopic cycles were proved on the Hamming
graph and on the hypercube
\cite{MilosSengul2019,AdamczakKotowskiMilos2021,KoteckyMilosUeltschi2016}. Very recently,
Elboim and Sly proved the existence of infinite cycles for the interchange process on
\(\mathbb Z^d\) for all \(d\ge 5\) at sufficiently large parameter, confirming a conjecture of
T\'oth in these dimensions \cite{ElboimSly2024}. Finally, and most relevant for the present
work, Poudevigne--Auboiron established macroscopic cycles for the interchange and quantum
Heisenberg models on random regular graphs by combining a deterministic drift argument with a
small-set sparsity estimate \cite{PoudevigneAuboiron2022}.
The purpose of this paper is to show that the deterministic drift argument offers significantly more flexibility than in the setting of \cite{PoudevigneAuboiron2022}. On the one hand, we extend the results from the random interchange to the general loop model with crosses and bars. From the model point of view, this is a generalization of
the interchange / Heisenberg setting. In the pure interchange case, a local insertion either
merges two cycles or splits one cycle. In more general loop models with reversals, however, a
local insertion on the same loop may also leave the number of loops unchanged while
restructuring the loop internally; (see, e.g.,~\cite{BjornbergKotowskiLeesMilos2019}).
We use a specific deterministic implementation of this split-versus-twist mechanism to derive the differential identities and drift inequalities that underpin the deterministic drift argument.
On the other hand, we show that the method is applicable to a much wider class of random graphs. The main
conceptual point of the paper is that the deterministic part of the argument can be formulated
on an arbitrary finite graph, and that the random-graph input enters only through a small-set
sparsity condition. Roughly speaking, one needs that sufficiently small vertex sets span at most
linearly many internal edges. Once this condition is available, together with a lower bound on
the asymptotic edge density, the deterministic drift estimate yields a positive lower bound on
the probability of a macroscopic loop. In this sense, the method applies to a whole class of
sparse random graphs rather than to a single ensemble.
We then verify the assumptions of this abstract criterion for three standard examples: random
regular graphs, sparse Erd\H{o}s--R\'enyi graphs, and simple bounded-degree configuration
models. This gives averaged lower bounds for the probability of a macroscopic loop in all three
cases. Thus our results generalize the random-regular-graph approach of
\cite{PoudevigneAuboiron2022} in two directions at once: first, from interchange-type models to
the full random loop model with crosses and bars; second, from one specific sparse random graph
ensemble to a general sparsity-based framework that covers several canonical classes of sparse
random graphs.
For integer values of the loop weight, one can go further. In that case the partition function
admits the standard finite-dimensional trace representation from the quantum-spin formalism, and
hence is log-convex in the time parameter \cite{Ueltschi2013}. This allows us to strengthen the
averaged bounds to pointwise-in-time statements away from the natural threshold time determined
by the drift inequality.
This shows, first, that the drift philosophy behind macroscopic-cycle results on sparse random
graphs is robust enough to survive the substantially richer local geometry of the
crosses-and-bars model. Second, it isolates a graph-theoretic sparsity principle that is
flexible and applies far beyond random regular graphs. Finally, it provides a natural framework
in which probabilistic ideas from interchange processes, random loop models, and quantum spin
systems can be treated simultaneously on sparse random structures.

\paragraph{Outline of the remaining sections.}
The paper is organized as follows. In Section~2 we introduce the random loop model and the random graph ensembles considered in this paper, and we state the main results. Section~3 develops the deterministic split--merge--rewire analysis and derives the associated differential identity and slice estimate. Section~4 verifies the required small-set sparsity condition for the random graph ensembles considered here. Section~5 proves the averaged and pointwise criteria and derives the corresponding results for the concrete examples. Finally, the appendix establishes log-convexity of the partition function for integer loop weights.

\section{Model and main results}
\subsection{The random loop model}

Let \(G=(V,E)\) be a finite simple graph, and write \(n:=|V|\). We consider the cylinder
\[
G\times S^1,
\qquad S^1:=\mathbb R/\mathbb Z .
\]
For each edge \(e\in E\), let \(\omega_e\) be an independent marked Poisson point process on
\(S^1\times\{\times,\mid\}\) with intensity measure
\[
t\,\dd s\otimes\bigl(u\,\delta_{\times}+(1-u)\,\delta_{\mid}\bigr),
\qquad t\ge 0,\quad u\in[0,1].
\]
We write \(\omega:=(\omega_e)_{e\in E}\) for the resulting marked configuration and denote its law by
\(\rho^G_{t,u}\). An atom \((s,\times)\in\omega_e\) is called a \emph{cross}, and an atom
\((s,\mid)\in\omega_e\) a \emph{bar}.
Given \(\omega\), the loops are defined as follows. Consider a particle moving in
\(G\times S^1\times\{-1,+1\}\), where the third coordinate records the vertical direction.
Between marks, the particle moves vertically with unit speed in the current direction.
Suppose that at time \(s\in S^1\) the particle is at a vertex \(x\in V\), and that there is a mark
on an edge \(e=\{x,y\}\). Then the particle instantaneously jumps from \(x\) to \(y\). If the mark
is a cross, the direction is preserved; if it is a bar, the direction is reversed. Since
\(S^1=\mathbb R/\mathbb Z\), this evolution is periodic in time, and every trajectory is closed.
The resulting collection of closed trajectories in \(G\times S^1\) is denoted by \(\cL(\omega)\).
We write \(\lambda(\omega):=|\cL(\omega)|\) for the number of loops. For a loop
\(\gamma\in\cL(\omega)\), we define its vertex support by
\[
\operatorname{supp}_V(\gamma)
:=
\{x\in V:\exists s\in S^1\text{ such that }(x,s)\in\gamma\},
\]
and, for \(s\in S^1\), its slice at height \(s\) by $S_\gamma(s):=\{x\in V:(x,s)\in\gamma\}$.
For \(\theta>0\), we define the \(\theta\)-weighted random loop measure by
\[
\Pbb^G_{\theta,t,u}(\dd\omega)
:=
\frac{1}{Z_G(\theta,t,u)}\,\theta^{\lambda(\omega)}\,\rho^G_{t,u}(\dd\omega),
\]
where the partition function is $Z_G(\theta,t,u):=\int \theta^{\lambda(\omega)}\,\rho^G_{t,u}(\dd\omega)$.
Expectations with respect to \(\Pbb^G_{\theta,t,u}\) are denoted by
\(\Ebb^G_{\theta,t,u}\).
For \(\eta>0\), the main event of interest is
\[
A_\eta
:=
\Bigl\{\exists \gamma\in\cL(\omega): |\operatorname{supp}_V(\gamma)|>\eta n\Bigr\}.
\]
Our graph-theoretic input is the following small-set sparsity condition: there exist
\(\eta,\varepsilon>0\) such that
\begin{equation}
e_G(S)\le (1+\varepsilon)|S|
\qquad
\text{for every }S\subset V\text{ with }|S|\le \eta |V|.
\tag{A1}\label{eq:A1}
\end{equation}
Here \(e_G(S)\) denotes the number of edges induced by \(S\).
\subsection{Random graph ensembles}

In the main results, the underlying graph is itself random. More precisely, for each \(n\),
we first sample a random finite graph \(G_n\) on vertex set \([n]:=\{1,\dots,n\}\), and then,
conditional on \(G_n\), we consider the random loop model on \(G_n\times S^1\) introduced
above. Thus \(\Pbb^{G_n}_{\theta,t,u}\) denotes the loop measure conditional on the graph,
whereas \(\Pbb\) denotes the law of the random graph itself.
We will consider the following three standard sparse random graph ensembles.

\begin{enumerate}[label=\textnormal{(\roman*)}]
    \item \textbf{Random regular graphs.}
    Fix \(d\ge 3\). A random \(d\)-regular graph on \([n]\) is a graph chosen uniformly from
    the set of all simple graphs on \([n]\) in which every vertex has degree exactly \(d\).
    We denote its law by \(p_{d,n}\). As usual, this model is defined only when \(nd\) is even.

    \item \textbf{Sparse Erd\H{o}s--R\'enyi graphs.}
    Fix \(\lambda>0\). In the Erd\H{o}s--R\'enyi model \(G(n,\lambda/n)\), each edge
    \(\{i,j\}\), \(1\le i<j\le n\), is present independently with probability \(\lambda/n\).
    Thus the expected degree of a vertex is asymptotically \(\lambda\), so this is a sparse
    random graph model.

    \item \textbf{Simple bounded-degree configuration models.}
    Let \((d_i^{(n)})_{i=1}^n\) be a deterministic degree sequence with even total degree.
    In the configuration model, one attaches \(d_i^{(n)}\) half-edges to vertex \(i\), and then
    chooses a uniform random pairing of all half-edges. This produces a random multigraph
    \(\widetilde G_n\), which may contain self-loops and multiple edges. The corresponding
    simple configuration model is obtained by conditioning on the event that
    \(\widetilde G_n\) is simple; we denote the resulting simple graph by \(G_n\).
\end{enumerate}

\subsection{Main results}

We now state our results for a general sequence of sparse random graphs.
The concrete graph ensembles introduced above will then appear as corollaries.
For \(\eta,\varepsilon>0\), define the event
\[
\mathcal E_n(\eta,\varepsilon)
:=
\Bigl\{
\forall S\subset V(G_n)\text{ with }|S|\le \eta n:\ 
e_{G_n}(S)\le (1+\varepsilon)|S|
\Bigr\}.
\]
Thus \(\mathcal E_n(\eta,\varepsilon)\) is the event that \(G_n\) satisfies the
small-set sparsity condition \eqref{eq:A1} up to scale \(\eta n\).
For \(u\in[0,1]\) and \(\theta>0\), write $c_{\theta,u}:=1+\theta\max\{u,1-u\}$.

\begin{assumption}[Sparse random graph assumptions]
\label{ass:sparse-random-graph}
Let \(G_n\) be a random simple graph on \([n]\). We assume that:
\begin{enumerate}[label=\textnormal{(\alph*)},ref=\textnormal{(\alph*)}]
    \item\label{ass:sparse-random-graph-a}
    for every \(\varepsilon>0\) there exists \(\eta>0\) such that
    \[
    \Pbb\bigl(\mathcal E_n(\eta,\varepsilon)\bigr)\xrightarrow[n\to\infty]{}1;
    \]

    \item\label{ass:sparse-random-graph-b}
    there exists \(\alpha>0\) such that
    \[
    \frac{|E(G_n)|}{n}\xrightarrow[n\to\infty]{\Pbb}\alpha.
    \]
\end{enumerate}
\end{assumption}

\begin{theorem}[Averaged macroscopic-loop criterion]
\label{thm:abstract-averaged}
Let \(G_n\) be a random simple graph on \([n]\). Fix \(u\in[0,1]\) and \(\theta>0\).
Assume that \((G_n)_n\) satisfies Assumption~\ref{ass:sparse-random-graph}, with limiting
edge density \(\alpha\), and that
\[
\alpha>c_{\theta,u}.
\]
Then there exist \(\eta>0\), \(s>0\), and \(c>0\) such that for every \(a\ge 0\),
\[
\Pbb\!\left(
\frac1s\int_a^{a+s}\Pbb^{G_n}_{\theta,t,u}(A_\eta)\,\dd t \ge c
\right)\xrightarrow[n\to\infty]{}1.
\]
\end{theorem}

\begin{remark}
The threshold condition \(\alpha>c_{\theta,u}\) in
Theorem~\ref{thm:abstract-averaged} should be understood as a proof threshold
arising from the drift method, not necessarily as the true critical threshold
of the model.
\end{remark}

\begin{corollary}[Existence of a good time]
\label{cor:abstract-good-time}
Under the assumptions of Theorem~\ref{thm:abstract-averaged}, there exist
\(\eta>0\), \(s>0\), and \(c>0\) such that for every \(a\ge 0\),
\[
\Pbb\!\left(
\exists\, t\in[a,a+s]:
\quad
\Pbb^{G_n}_{\theta,t,u}(A_\eta)\ge c
\right)\xrightarrow[n\to\infty]{}1.
\]
\end{corollary}
For integer loop weights, the partition function admits a trace representation.
This yields a stronger pointwise-in-time statement.
\begin{theorem}[Pointwise macroscopic-loop criterion]
\label{thm:abstract-pointwise}
Let \(G_n\) be a random simple graph on \([n]\). Fix \(u\in[0,1]\) and
\(\theta\in\mathbb N\) with \(\theta>1\). Assume that \((G_n)_n\) satisfies
Assumption~\ref{ass:sparse-random-graph}, with limiting edge density \(\alpha\), and that $\alpha>c_{\theta,u}$.
Then for every \(\delta>0\) there exist \(\eta>0\) and \(c>0\) such that for every
\[
t\ge \frac{\theta\log\theta}{(\theta-1)(\alpha-c_{\theta,u})}+\delta,
\]
one has
\[
\Pbb\!\left(
\Pbb^{G_n}_{\theta,t,u}(A_\eta)\ge c
\right)\xrightarrow[n\to\infty]{}1.
\]
\end{theorem}
The three graph ensembles introduced above satisfy the assumptions of
Theorems~\ref{thm:abstract-averaged} and \ref{thm:abstract-pointwise}. We
therefore obtain the following concrete consequences.

\begin{corollary}[Examples: averaged macroscopic-loop bounds]
\label{cor:examples-averaged}
Fix \(u\in[0,1]\) and \(\theta>0\).

\begin{enumerate}[label=\textnormal{(\roman*)}]
    \item \textbf{Random regular graphs.}
    Let \(G_n\) be the random \(d\)-regular graph on \([n]\), with law \(p_{d,n}\).
    If $d>2c_{\theta,u}$, then there exist \(\eta>0\), \(s>0\), and \(c>0\) such that for every \(a\ge 0\),
    \[
    p_{d,n}\!\left(
    \frac1s\int_a^{a+s}\Pbb^{G_n}_{\theta,t,u}(A_\eta)\,\dd t \ge c
    \right)\xrightarrow[n\to\infty]{}1.
    \]

    \item \textbf{Sparse Erd\H{o}s--R\'enyi graphs.}
    Let \(G_n\sim G(n,\lambda/n)\). If $\lambda>2c_{\theta,u}$,
    then there exist \(\eta>0\), \(s>0\), and \(c>0\) such that for every \(a\ge 0\),
    \[
    \Pbb\!\left(
    \frac1s\int_a^{a+s}\Pbb^{G_n}_{\theta,t,u}(A_\eta)\,\dd t \ge c
    \right)\xrightarrow[n\to\infty]{}1.
    \]

    \item \textbf{Simple bounded-degree configuration model.}
    Let \(G_n\) be the simple graph obtained from the configuration model with
    degree sequence \((d_i^{(n)})_{i=1}^n\). Assume that
    \[
    1\le d_i^{(n)}\le \Delta
    \qquad\text{for all }i,n,
    \]
    and $\frac1n\sum_{i=1}^n d_i^{(n)}\to \rho$. If $\rho>2c_{\theta,u}$,
    then there exist \(\eta>0\), \(s>0\), and \(c>0\) such that for every \(a\ge 0\),
    \[
    \Pbb\!\left(
    \frac1s\int_a^{a+s}\Pbb^{G_n}_{\theta,t,u}(A_\eta)\,\dd t \ge c
    \right)\xrightarrow[n\to\infty]{}1.
    \]
\end{enumerate}
\end{corollary}

\begin{remark}
In the three cases of Corollary~\ref{cor:examples-averaged}, the asymptotic edge density is
\[
\alpha=\frac d2,\qquad \alpha=\frac{\lambda}{2},\qquad \alpha=\frac{\rho}{2},
\]
respectively.
\end{remark}
For \(m>2c_{\theta,u}\) and \(\theta>1\), define $T_{\theta,u}(m):=
\frac{\theta\log\theta}{(\theta-1)\bigl(\frac m2-c_{\theta,u}\bigr)}$.
\begin{corollary}[Examples: pointwise macroscopic-loop bounds]
\label{cor:examples-pointwise}
Fix \(u\in[0,1]\) and \(\theta\in\mathbb N\) with \(\theta>1\).

\begin{enumerate}[label=\textnormal{(\roman*)}]
    \item \textbf{Random regular graphs.}
    Let \(G_n\) be the random \(d\)-regular graph on \([n]\).
    If $d>2c_{\theta,u}$, then for every \(\delta>0\) there exist \(\eta>0\) and \(c>0\) such that for every $t\ge T_{\theta,u}(d)+\delta$, one has
    \[
    p_{d,n}\!\left(
    \Pbb^{G_n}_{\theta,t,u}(A_\eta)\ge c
    \right)\xrightarrow[n\to\infty]{}1.
    \]

    \item \textbf{Sparse Erd\H{o}s--R\'enyi graphs.}
    Let \(G_n\sim G(n,\lambda/n)\). If $\lambda>2c_{\theta,u}$, then for every \(\delta>0\) there exist \(\eta>0\) and \(c>0\) such that for every $t\ge T_{\theta,u}(\lambda)+\delta$, one has
    \[
    \Pbb\!\left(
    \Pbb^{G_n}_{\theta,t,u}(A_\eta)\ge c
    \right)\xrightarrow[n\to\infty]{}1.
    \]

    \item \textbf{Simple bounded-degree configuration model.}
    Let \(G_n\) be the simple graph obtained from the configuration model with
    degree sequence \((d_i^{(n)})_{i=1}^n\). Assume that
    \[
    1\le d_i^{(n)}\le \Delta
    \qquad\text{for all }i,n,
    \]
    and $\frac1n\sum_{i=1}^n d_i^{(n)}\to \rho$. If $\rho>2c_{\theta,u}$,
    then for every \(\delta>0\) there exist \(\eta>0\) and \(c>0\) such that for every  $t\ge T_{\theta,u}(\rho)+\delta$,
    one has
    \[
    \Pbb\!\left(
    \Pbb^{G_n}_{\theta,t,u}(A_\eta)\ge c
    \right)\xrightarrow[n\to\infty]{}1.
    \]
\end{enumerate}
\end{corollary}

\begin{remark}
The small-set sparsity assumption is also satisfied, in a trivial way, by tree-like models.
For instance, if \(G_n\) is a conditioned Galton--Watson tree on \(n\) vertices, then every
induced subgraph is a forest, and hence
\[
e_{G_n}(S)\le |S|-1
\qquad\text{for every }S\subset V(G_n).
\]
Thus \(\mathcal E_n(\eta,\varepsilon)\) holds deterministically for every \(\eta\le 1\) and every
\(\varepsilon>0\). However,
\[
\frac{|E(G_n)|}{n}=\frac{n-1}{n}\to 1,
\]
so the criterion does not yield macroscopic loops in this case, since it requires the asymptotic
edge density to be strictly larger than \(c_{\theta,u}>1\).
\end{remark}

\section{Deterministic mechanism}

The main deterministic input is the following reduction of the random loop model
to the graph-theoretic condition \eqref{eq:A1}. We state it first and prove it at
the end of this section.

\begin{proposition}[Deterministic drift bound under small-set sparsity]
\label{prop:deterministic-drift}
Let \(G=(V,E)\) be a finite graph with \(|V|=n\), and assume that \eqref{eq:A1}
holds for some \(\eta,\varepsilon>0\). Then
\begin{equation}
D_{\theta,u}(t)
\le
\bigl(c_{\theta,u}(1+\varepsilon)n-|E|\bigr)
+
\bigl((1+\theta)|E|-c_{\theta,u}(1+\varepsilon)n\bigr)\,
\Pbb^G_{\theta,t,u}(A_\eta),
\label{eq:deterministic-drift}
\end{equation}
where
\[
D_{\theta,u}(t)=
\begin{cases}
\dfrac{\dd}{\dd t}\,\Ebb^G_{1,t,u}\bigl[\lambda(\omega)\bigr], & \theta=1,\\[1.2ex]
\dfrac{\theta}{\theta-1}\dfrac{\dd}{\dd t}\log Z_G(\theta,t,u), & \theta\neq 1.
\end{cases}
\]
\end{proposition}

\subsection{Local split--merge--rewire structure}

We begin by formulating the local split--merge--rewire structure. Define
\[
T(\omega):=\{s\in S^1:\text{there is at least one mark of }\omega\text{ at time }s\}.
\]
Since \(\omega\) is finite, the set \(T(\omega)\) is finite. We call
\(s\in S^1\setminus T(\omega)\) a \emph{regular time}.
At every regular time \(s\in S^1\setminus T(\omega)\), each loop of \(\cL(\omega)\)
meets \(V\times\{s\}\) transversally. Hence, for every point \((x,s)\) lying on a
loop, the loop has a well-defined local vertical orientation at \((x,s)\), which we
identify with a sign in \(\{-1,+1\}\).
Now fix a regular time \(s\in S^1\setminus T(\omega)\) and an edge \(e=\{x,y\}\in E\).
Then exactly one of the following three possibilities occurs:
\begin{enumerate}[label=(\alph*)]
    \item \((x,s)\) and \((y,s)\) belong to distinct loops;
    \item \((x,s)\) and \((y,s)\) belong to the same loop and have the same local
    vertical orientation;
    \item \((x,s)\) and \((y,s)\) belong to the same loop and have opposite local
    vertical orientations.
\end{enumerate}
Accordingly, we define
\[
I_+(\omega;e,s)
:=
\1_{\{(x,s)\text{ and }(y,s)\text{ belong to the same loop with the same local vertical orientation}\}},
\]
\[
I_-(\omega;e,s)
:=
\1_{\{(x,s)\text{ and }(y,s)\text{ belong to the same loop with opposite local vertical orientations}\}},
\]
and
\[
I_{\neq}(\omega;e,s):=1-I_+(\omega;e,s)-I_-(\omega;e,s).
\]
At irregular times \(s\in T(\omega)\), we set all three indicators equal to \(0\).
Since \(T(\omega)\) is finite, this convention does not affect any of the integral
identities below.
We then define
\[
J_+(\omega):=\sum_{e\in E}\int_{S^1} I_+(\omega;e,s)\,\dd s,
\qquad
J_-(\omega):=\sum_{e\in E}\int_{S^1} I_-(\omega;e,s)\,\dd s.
\]

\begin{lemma}[Local split--merge--rewire rule]
\label{lem:local-surgery}
Let \(\omega\) be a configuration, let \(\cL(\omega)\) be the associated loop
configuration, let \(e=\{x,y\}\in E\), and let \(s\in S^1\setminus T(\omega)\) be a
regular time. Consider the effect of inserting at \((e,s)\) either a cross or a bar.

\begin{enumerate}[label=(\roman*)]
    \item If \(I_{\neq}(\omega;e,s)=1\), then inserting either a cross or a bar
    merges the two loops containing \((x,s)\) and \((y,s)\) into one.

    \item If \(I_+(\omega;e,s)=1\), then inserting a cross splits the loop into two,
    whereas inserting a bar rewires the loop without changing the number of loops.

    \item If \(I_-(\omega;e,s)=1\), then inserting a bar splits the loop into two,
    whereas inserting a cross rewires the loop without changing the number of loops.
\end{enumerate}
\end{lemma}

\begin{proof}
Take disjoint small neighbourhoods of \((x,s)\) and \((y,s)\) and remove the
vertical strands there. This cuts the configuration locally into four half-strands,
while leaving it unchanged away from these neighbourhoods.
If \(I_{\neq}(\omega;e,s)=1\), then \((x,s)\) and \((y,s)\) lie on two distinct loops.
Cutting at these two points opens the two loops into two open strands, and inserting
either a cross or a bar reconnects these strands into a single loop. Hence the two
loops merge.
Assume now that \((x,s)\) and \((y,s)\) lie on the same loop. Cutting at the two
points decomposes this loop into two complementary oriented arcs.
If \(I_+(\omega;e,s)=1\), then the local vertical orientations at \((x,s)\) and
\((y,s)\) agree. A cross reconnects the half-strands in the orientation-preserving
way, so that each complementary arc closes separately; hence the original loop splits
into two loops. By contrast, a bar reverses the orientation at one endpoint and
reconnects the two arcs crosswise, producing a single loop with changed geometry
but unchanged loop number.
If \(I_-(\omega;e,s)=1\), then the local vertical orientations at \((x,s)\) and
\((y,s)\) are opposite. The roles of crosses and bars are therefore reversed: a bar
closes the two complementary arcs separately and hence splits the loop, whereas a
cross reconnects them crosswise and merely rewires the loop.
\end{proof}

\subsection{Differential identity}

The next lemma is a simple finite-intensity special case of standard perturbation
formulas for Poisson processes; see, for instance,
\cite[Section~6]{Last2014Perturbation}. For completeness we give a short direct proof.

\begin{lemma}[Differentiation formula for a finite-intensity Poisson process]
\label{lem:poisson-diff}
Let \((\mathsf X,\nu)\) be a finite measure space, and for \(t\ge 0\) let \(\Pi_t\) be
a Poisson point process on \(\mathsf X\) with intensity measure \(t\nu\). Then, for
every finite counting measures, the map
\(t\mapsto \Ebb[F(\Pi_t)]\) is differentiable on \((0,\infty)\), and its right
derivative at \(t=0\) is given by
\[
\frac{\dd}{\dd t}\Ebb[F(\Pi_t)]
=
\int_{\mathsf X}\Ebb\!\left[F(\Pi_t+\delta_z)-F(\Pi_t)\right]\nu(\dd z).
\]
\end{lemma}

\begin{proof}
If \(\nu(\mathsf X)=0\), then \(\Pi_t\) is almost surely empty for every \(t\), and
both sides vanish identically. Thus assume \(\nu(\mathsf X)>0\).

Fix \(t\ge 0\), and let \(\Xi_h\) be an independent Poisson point process on
\(\mathsf X\) with intensity \(h\nu\), \(h>0\). Then
\[
\Pi_{t+h}\overset{d}{=}\Pi_t+\Xi_h,
\]
and therefore
\[
\Ebb[F(\Pi_{t+h})]-\Ebb[F(\Pi_t)]
=
\Ebb\!\left[F(\Pi_t+\Xi_h)-F(\Pi_t)\right].
\]
Since \(\nu(\mathsf X)<\infty\), the total number \(N_h:=\Xi_h(\mathsf X)\) is
Poisson with mean \(h\nu(\mathsf X)\). In particular,
\[
\Pbb(N_h=0)=1-h\nu(\mathsf X)+O(h^2),
\qquad
\Pbb(N_h=1)=h\nu(\mathsf X)+O(h^2),
\qquad
\Pbb(N_h\ge 2)=O(h^2).
\]
Because \(F\) is bounded, the contribution of the event \(\{N_h\ge 2\}\) is
\(O(h^2)\).

On the event \(\{N_h=1\}\), the unique atom \(Z_h\) has distribution
\(\nu/\nu(\mathsf X)\). Hence
\[
\Ebb\!\left[F(\Pi_t+\Xi_h)-F(\Pi_t)\,\middle|\,N_h=1\right]
=
\frac{1}{\nu(\mathsf X)}
\int_{\mathsf X}
\Ebb\!\left[F(\Pi_t+\delta_z)-F(\Pi_t)\right]\nu(\dd z).
\]
Combining the three cases, we obtain
\[
\Ebb[F(\Pi_{t+h})]-\Ebb[F(\Pi_t)]
=
h\int_{\mathsf X}
\Ebb\!\left[F(\Pi_t+\delta_z)-F(\Pi_t)\right]\nu(\dd z)
+O(h^2).
\]
Dividing by \(h\) and letting \(h\downarrow 0\) gives the claim.
\end{proof}

\begin{proposition}[Differential identity]
\label{prop:exact-drift}
Let \(G=(V,E)\) be a finite graph, let \(u\in[0,1]\), and let \(\theta>0\). Then
\[
D_{\theta,u}(t)
=
\Ebb_{\theta,t,u}^G\!\left[
(1+\theta u)J_+ + (1+\theta(1-u))J_- - |E|
\right],
\]
where
\[
D_{\theta,u}(t)=
\begin{cases}
\dfrac{\dd}{\dd t}\,\Ebb_{1,t,u}^G[\lambda(\omega)], & \theta=1,\\[1.2ex]
\dfrac{\theta}{\theta-1}\dfrac{\dd}{\dd t}\log Z_G(\theta,t,u), & \theta\neq 1.
\end{cases}
\]
\end{proposition}

\begin{proof}
We encode the model as a marked Poisson point process on
\[
\mathsf X:=E\times S^1\times\{\times,\mid\}.
\]
Let \(\nu_u\) be the finite measure on \(\mathsf X\) defined by
\[
\nu_u(\dd e,\dd s,\dd m)
:=
\sum_{e\in E}\dd s\,
\bigl(u\,\delta_{\times}+(1-u)\,\delta_{\mid}\bigr)(\dd m).
\]
Thus \(\nu_u(\mathsf X)=|E|\), and the underlying random loop configuration is a
Poisson point process \(\omega\) on \(\mathsf X\) with intensity measure \(t\nu_u\).
Its law is denoted by \(\rho_{t,u}^G\).

For \(z=(e,s,m)\in\mathsf X\), we write \(\omega+\delta_z\) for the configuration
obtained by inserting one additional marked point \(z\). Since
\(\lambda(\omega)\in\{1,\dots,|V|\}\), the functionals
\[
F_\theta(\omega):=\theta^{\lambda(\omega)}
\qquad\text{and}\qquad
F_1(\omega):=\lambda(\omega)
\]
are bounded.
We first treat the case \(\theta\neq 1\). By Lemma~\ref{lem:poisson-diff},
\[
\frac{\dd}{\dd t}Z_G(\theta,t,u)
=
\int_{\mathsf X}
\Ebb_{\rho_{t,u}^G}\!\left[
\theta^{\lambda(\omega+\delta_z)}-\theta^{\lambda(\omega)}
\right]\nu_u(\dd z).
\]
Write \(z=(e,s,m)\). Since the exceptional set \(T(\omega)\subset S^1\) is finite,
the identities from Lemma~\ref{lem:local-surgery} hold for \(\nu_u\)-almost every
\((e,s,m)\).
If \(m=\times\), then Lemma~\ref{lem:local-surgery} implies
\[
\lambda(\omega+\delta_{(e,s,\times)})-\lambda(\omega)
=
I_+(\omega;e,s)-I_{\neq}(\omega;e,s).
\]
Likewise, if \(m=\mid\), then
\[
\lambda(\omega+\delta_{(e,s,\mid)})-\lambda(\omega)
=
I_-(\omega;e,s)-I_{\neq}(\omega;e,s).
\]
Hence, for \(\nu_u\)-almost every \((e,s,m)\),
\[
\theta^{\lambda(\omega+\delta_{(e,s,\times)})}-\theta^{\lambda(\omega)}
=
\theta^{\lambda(\omega)}
\Bigl((\theta-1)I_+ + (\theta^{-1}-1)I_{\neq}\Bigr),
\]
and
\[
\theta^{\lambda(\omega+\delta_{(e,s,\mid)})}-\theta^{\lambda(\omega)}
=
\theta^{\lambda(\omega)}
\Bigl((\theta-1)I_- + (\theta^{-1}-1)I_{\neq}\Bigr).
\]
Integrating over the mark distribution
\(u\delta_\times+(1-u)\delta_{\mid}\), we get
\[
\int_{\{\times,\mid\}}
\bigl(\theta^{\lambda(\omega+\delta_{(e,s,m)})}-\theta^{\lambda(\omega)}\bigr)
\bigl(u\delta_\times+(1-u)\delta_{\mid}\bigr)(\dd m)
\]
\[
=
\theta^{\lambda(\omega)}
\Bigl(
(\theta-1)u\,I_+ + (\theta-1)(1-u)\,I_- + (\theta^{-1}-1)I_{\neq}
\Bigr).
\]
Now note that for almost every \((e,s)\),
\[
I_+(\omega;e,s)+I_-(\omega;e,s)+I_{\neq}(\omega;e,s)=1,
\]
because exactly one of the three alternatives above holds. Therefore
\[
\sum_{e\in E}\int_{S^1} I_{\neq}(\omega;e,s)\,\dd s
=
|E|-J_+(\omega)-J_-(\omega).
\]
Using also
\[
\sum_{e\in E}\int_{S^1} I_+(\omega;e,s)\,\dd s = J_+(\omega),
\qquad
\sum_{e\in E}\int_{S^1} I_-(\omega;e,s)\,\dd s = J_-(\omega),
\]
we obtain
\[
\frac{\dd}{\dd t}Z_G(\theta,t,u)
=
\Ebb_{\rho_{t,u}^G}\!\left[
\theta^{\lambda(\omega)}
\Bigl(
(\theta-1)u\,J_+
+
(\theta-1)(1-u)\,J_-
+
(\theta^{-1}-1)(|E|-J_+-J_-)
\Bigr)
\right].
\]
Since \(\theta^{-1}-1=-(\theta-1)/\theta\), the bracket equals
\[
\frac{\theta-1}{\theta}
\Bigl(
(1+\theta u)J_+ + (1+\theta(1-u))J_- - |E|
\Bigr).
\]
Therefore
\[
\frac{\dd}{\dd t}Z_G(\theta,t,u)
=
\frac{\theta-1}{\theta}\,
\Ebb_{\rho_{t,u}^G}\!\left[
\theta^{\lambda(\omega)}
\Bigl(
(1+\theta u)J_+ + (1+\theta(1-u))J_- - |E|
\Bigr)
\right].
\]
Dividing by \(Z_G(\theta,t,u)\) yields
\[
\frac{\theta}{\theta-1}\frac{\dd}{\dd t}\log Z_G(\theta,t,u)
=
\Ebb_{\theta,t,u}^G\!\left[
(1+\theta u)J_+ + (1+\theta(1-u))J_- - |E|
\right].
\]
For \(\theta=1\), apply Lemma~\ref{lem:poisson-diff} to the bounded functional
\(F(\omega)=\lambda(\omega)\). Then
\[
\frac{\dd}{\dd t}\Ebb_{1,t,u}^G[\lambda(\omega)]
=
\int_{\mathsf X}
\Ebb_{1,t,u}^G\!\left[
\lambda(\omega+\delta_z)-\lambda(\omega)
\right]\nu_u(\dd z).
\]
By Lemma~\ref{lem:local-surgery},
\[
\lambda(\omega+\delta_{(e,s,\times)})-\lambda(\omega)=I_+-I_{\neq},
\]
\[
\lambda(\omega+\delta_{(e,s,\mid)})-\lambda(\omega)=I_--I_{\neq}
\]
for \(\nu_u\)-almost every \((e,s)\). Averaging over the mark distribution gives
\[
u(I_+-I_{\neq})+(1-u)(I_--I_{\neq})
=
uI_+ + (1-u)I_- - I_{\neq}.
\]
Using \(I_{\neq}=1-I_+-I_-\), this becomes
\[
(1+u)I_+ + (2-u)I_- - 1.
\]
Integrating over \(e\) and \(s\) yields
\[
\frac{\dd}{\dd t}\Ebb_{1,t,u}^G[\lambda(\omega)]
=
\Ebb_{1,t,u}^G\!\left[
(1+u)J_+ + (2-u)J_- - |E|
\right].
\]
Since
\[
1+u=1+\theta u,
\qquad
2-u=1+\theta(1-u)
\qquad\text{when }\theta=1,
\]
this is exactly the claimed formula.
\end{proof}

\subsection{Slice identity}

We now express the same-loop insertion volume in a fixed time slice in terms of
induced edge counts of the corresponding slice sets.
Let \(\omega\) be a configuration. For a regular time \(s\in S^1\setminus T(\omega)\),
define for each edge \(e=\{x,y\}\in E\),
\[
I_{\mathrm{same}}(\omega;e,s)
:=
\1_{\{(x,s)\text{ and }(y,s)\text{ belong to the same loop}\}}.
\]
At irregular times \(s\in T(\omega)\), define \(I_{\mathrm{same}}(\omega;e,s):=0\).

\begin{lemma}[Pointwise slice identity]
\label{lem:pointwise-slice}
Let \(\omega\) be a configuration, and let \(s\in S^1\setminus T(\omega)\) be a regular time. Then
\begin{equation}
I_{\mathrm{same}}(\omega;e,s)=I_+(\omega;e,s)+I_-(\omega;e,s)
\qquad\text{for every }e\in E,
\label{eq:same-split}
\end{equation}
and
\begin{equation}
\sum_{e\in E} I_{\mathrm{same}}(\omega;e,s)
=
\sum_{\gamma\in\cL(\omega)} e_G(S_\gamma(s)).
\label{eq:pointwise-slice}
\end{equation}
\end{lemma}

\begin{proof}
Fix a regular time \(s\). Since \(s\notin T(\omega)\), no mark is present at height \(s\). Therefore, for every
edge \(e=\{x,y\}\), if \((x,s)\) and \((y,s)\) belong to the same loop, then the loop
has well-defined local vertical orientations at both points. These orientations are
either equal or opposite, and the two possibilities are mutually exclusive. This proves
\eqref{eq:same-split}.
To prove \eqref{eq:pointwise-slice}, fix \(e=\{x,y\}\in E\). Since the loops in
\(\cL(\omega)\) partition \(V\times S^1\), there is a unique loop \(\gamma_x\)
containing \((x,s)\) and a unique loop \(\gamma_y\) containing \((y,s)\). Hence
\[
I_{\mathrm{same}}(\omega;e,s)=1
\quad\Longleftrightarrow\quad
\gamma_x=\gamma_y
\quad\Longleftrightarrow\quad
\exists\,\gamma\in\cL(\omega)\text{ such that }x,y\in S_\gamma(s).
\]
Since the sets \(S_\gamma(s)\), \(\gamma\in\cL(\omega)\), are pairwise disjoint at
fixed \(s\), there is at most one such loop \(\gamma\). Therefore
\[
I_{\mathrm{same}}(\omega;e,s)
=
\sum_{\gamma\in\cL(\omega)}
\1_{\{x\in S_\gamma(s)\}}\1_{\{y\in S_\gamma(s)\}}.
\]
Summing over all \(e=\{x,y\}\in E\), we obtain
\[
\sum_{e\in E} I_{\mathrm{same}}(\omega;e,s)
=
\sum_{\gamma\in\cL(\omega)}
\sum_{\{x,y\}\in E}
\1_{\{x\in S_\gamma(s)\}}\1_{\{y\in S_\gamma(s)\}}.
\]
The inner sum is exactly the number of edges of \(G\) induced by \(S_\gamma(s)\),
namely \(e_G(S_\gamma(s))\). Hence
\[
\sum_{e\in E} I_{\mathrm{same}}(\omega;e,s)
=
\sum_{\gamma\in\cL(\omega)} e_G(S_\gamma(s)),
\]
which proves \eqref{eq:pointwise-slice}.
\end{proof}

\begin{corollary}[Integrated slice identity]
\label{cor:integrated-slice}
For every configuration \(\omega\),
\[
J_+(\omega)+J_-(\omega)
=
\int_{S^1}\sum_{\gamma\in\cL(\omega)} e_G(S_\gamma(s))\,\dd s.
\]
\end{corollary}

\begin{proof}
By Lemma~\ref{lem:pointwise-slice}, for every regular time
\(s\in S^1\setminus T(\omega)\),
\[
\sum_{e\in E}\bigl(I_+(\omega;e,s)+I_-(\omega;e,s)\bigr)
=
\sum_{\gamma\in\cL(\omega)} e_G(S_\gamma(s)).
\]
Since \(T(\omega)\) is finite, the same identity holds for almost every \(s\in S^1\).
Integrating over \(s\) and using the definitions of \(J_+\) and \(J_-\), we obtain
\[
J_+(\omega)+J_-(\omega)
=
\int_{S^1}\sum_{\gamma\in\cL(\omega)} e_G(S_\gamma(s))\,\dd s.
\]
\end{proof}

\begin{lemma}
\label{lem:slice}
Assume that \eqref{eq:A1} holds for some \(\eta,\varepsilon>0\). Then on the event
\(A_\eta^c\),
\[
J_+(\omega)+J_-(\omega)\le (1+\varepsilon)n.
\]
\end{lemma}

\begin{proof}
On \(A_\eta^c\), every loop \(\gamma\in\cL(\omega)\) satisfies
\[
|\operatorname{supp}_V(\gamma)|\le \eta n.
\]
Hence, for every \(s\in S^1\),
\[
|S_\gamma(s)|\le |\operatorname{supp}_V(\gamma)|\le \eta n.
\]
By \eqref{eq:A1},
\[
e_G(S_\gamma(s))\le (1+\varepsilon)|S_\gamma(s)|.
\]
Using Corollary~\ref{cor:integrated-slice}, we get
\[
J_+ + J_-
=
\int_{S^1}\sum_{\gamma\in\cL(\omega)} e_G(S_\gamma(s))\,\dd s
\le
(1+\varepsilon)\int_{S^1}\sum_{\gamma\in\cL(\omega)} |S_\gamma(s)|\,\dd s.
\]
For each fixed \(s\), the sets \(S_\gamma(s)\), \(\gamma\in\cL(\omega)\), partition \(V\),
and therefore
\[
\sum_{\gamma\in\cL(\omega)} |S_\gamma(s)|=n.
\]
Thus
\[
J_+ + J_-
\le
(1+\varepsilon)\int_{S^1} n\,\dd s
=
(1+\varepsilon)n.
\]
\end{proof}

\begin{proof}[Proof of Proposition~\ref{prop:deterministic-drift}]
By Proposition~\ref{prop:exact-drift},
\[
D_{\theta,u}(t)
=
\Ebb_{\theta,t,u}^G\!\left[
(1+\theta u)J_+ + (1+\theta(1-u))J_- - |E|
\right].
\]
Since
\[
(1+\theta u)J_+ + (1+\theta(1-u))J_-
\le
c_{\theta,u}(J_+ + J_-),
\]
where
\[
c_{\theta,u}=1+\theta\max\{u,1-u\},
\]
we obtain
\[
D_{\theta,u}(t)
\le
\Ebb_{\theta,t,u}^G\!\left[c_{\theta,u}(J_+ + J_-) - |E|\right].
\]
Split according to the event \(A_\eta\):
\[
D_{\theta,u}(t)
\le
\Ebb_{\theta,t,u}^G\!\left[(c_{\theta,u}(J_+ + J_-) - |E|)\1_{A_\eta^c}\right]
+
\Ebb_{\theta,t,u}^G\!\left[(c_{\theta,u}(J_+ + J_-) - |E|)\1_{A_\eta}\right].
\]
On \(A_\eta^c\), Lemma~\ref{lem:slice} gives
\[
J_+ + J_- \le (1+\varepsilon)n.
\]
On the other hand, always
\[
J_+ + J_- \le |E|,
\]
since for each edge \(e\in E\), the time-circle has length \(1\). Hence
\[
D_{\theta,u}(t)
\le
\bigl(c_{\theta,u}(1+\varepsilon)n-|E|\bigr)\Pbb_{\theta,t,u}^G(A_\eta^c)
+
\bigl(c_{\theta,u}|E|-|E|\bigr)\Pbb_{\theta,t,u}^G(A_\eta).
\]
Rearranging,
\[
D_{\theta,u}(t)
\le
\bigl(c_{\theta,u}(1+\varepsilon)n-|E|\bigr)
+
c_{\theta,u}\bigl(|E|-(1+\varepsilon)n\bigr)\Pbb_{\theta,t,u}^G(A_\eta).
\]
Since \(c_{\theta,u}\le 1+\theta\), we also have
\[
c_{\theta,u}\bigl(|E|-(1+\varepsilon)n\bigr)
\le
(1+\theta)|E|-c_{\theta,u}(1+\varepsilon)n,
\]
and \eqref{eq:deterministic-drift} follows.
\end{proof}

\section{Verification of the small-set sparsity condition}

In this section we verify that the small-set sparsity assumption
\begin{equation}
e_G(S)\le (1+\varepsilon)|S|
\qquad
\text{for every }S\subset V(G)\text{ with }|S|\le \eta |V(G)|
\tag{A1}
\end{equation}
holds with high probability for the random graph models considered in this paper. For \(\eta,\varepsilon>0\), define the event
\[
\mathcal E_n(\eta,\varepsilon)
:=
\Bigl\{
\forall S\subset V(G_n)\text{ with }|S|\le \eta n:\ 
e_{G_n}(S)\le (1+\varepsilon)|S|
\Bigr\}.
\]

\subsection{Sparse Erd\H{o}s--R\'enyi graphs}

\begin{lemma}
\label{lem:A1-ER}
Fix \(\lambda>0\) and \(\varepsilon>0\). Let \(G_n\sim G(n,\lambda/n)\). Then there exists \(\eta=\eta(\lambda,\varepsilon)>0\) such that
\[
\Pbb\bigl(\mathcal E_n(\eta,\varepsilon)\bigr)\xrightarrow[n\to\infty]{}1.
\]
\end{lemma}

\begin{proof}
Fix \(S\subset[n]\) with \(|S|=k\). Then
\[
e_{G_n}(S)\sim \mathrm{Bin}\!\Bigl(\binom{k}{2},\frac{\lambda}{n}\Bigr).
\]
Set \(m:=(1+\varepsilon)k\). By the standard binomial tail bound
\[
\Pbb(\mathrm{Bin}(N,p)\ge m)\le \Bigl(\frac{eNp}{m}\Bigr)^m,
\]
we obtain
\[
\Pbb\bigl(e_{G_n}(S)\ge (1+\varepsilon)k\bigr)
\le
\left(
\frac{e\binom{k}{2}(\lambda/n)}{(1+\varepsilon)k}
\right)^{(1+\varepsilon)k}
\le
\left(
\frac{e\lambda}{2(1+\varepsilon)}\frac{k}{n}
\right)^{(1+\varepsilon)k}.
\]
Hence, by a union bound over all subsets \(S\subset[n]\) of size \(k\),
\[
\Pbb\Bigl(\exists S\subset[n],\ |S|=k:\ e_{G_n}(S)\ge (1+\varepsilon)k\Bigr)
\le
\binom{n}{k}
\left(
\frac{e\lambda}{2(1+\varepsilon)}\frac{k}{n}
\right)^{(1+\varepsilon)k}.
\]
Using \(\binom{n}{k}\le (en/k)^k\), this yields
\[
\Pbb\Bigl(\exists S\subset[n],\ |S|=k:\ e_{G_n}(S)\ge (1+\varepsilon)k\Bigr)
\le
\left(
C_{\lambda,\varepsilon}\Bigl(\frac{k}{n}\Bigr)^\varepsilon
\right)^k
\]
for a finite constant \(C_{\lambda,\varepsilon}>0\).
We now split the sum over \(k\) into two ranges.
If \(1\le k\le \log n\), then
\[
\left(
C_{\lambda,\varepsilon}\Bigl(\frac{k}{n}\Bigr)^\varepsilon
\right)^k
\le
\left(
C_{\lambda,\varepsilon}\Bigl(\frac{\log n}{n}\Bigr)^\varepsilon
\right)^k,
\]
and therefore
\[
\sum_{k=1}^{\lfloor \log n\rfloor}
\Pbb\Bigl(\exists S,\ |S|=k:\ e_{G_n}(S)\ge (1+\varepsilon)k\Bigr)
\xrightarrow[n\to\infty]{}0.
\]
Next choose \(\eta>0\) so small that
\[
C_{\lambda,\varepsilon}\eta^\varepsilon<e^{-2}.
\]
Then for every \(k\) with \(\log n\le k\le \eta n\),
\[
\left(
C_{\lambda,\varepsilon}\Bigl(\frac{k}{n}\Bigr)^\varepsilon
\right)^k
\le
\left(C_{\lambda,\varepsilon}\eta^\varepsilon\right)^k
\le e^{-2k}.
\]
Hence
\[
\sum_{k=\lceil \log n\rceil}^{\lfloor \eta n\rfloor}
\Pbb\Bigl(\exists S,\ |S|=k:\ e_{G_n}(S)\ge (1+\varepsilon)k\Bigr)
\le
\sum_{k=\lceil \log n\rceil}^{\infty} e^{-2k}
\xrightarrow[n\to\infty]{}0.
\]
Combining the two ranges proves that $\Pbb\bigl(\mathcal E_n(\eta,\varepsilon)^c\bigr)\to 0$.
\end{proof}

\subsection{Simple bounded-degree configuration model}

Let \(\widetilde G_n\) denote the multigraph given by the pairing construction with degree sequence \((d_i^{(n)})_{i=1}^n\), and let \(G_n\) denote the corresponding simple graph conditioned on the event that \(\widetilde G_n\) is simple.

\begin{lemma}
\label{lem:A1-CM-pairing}
Assume that
\[
1\le d_i^{(n)}\le \Delta
\qquad\text{for all }i,n,
\]
and let \(\widetilde G_n\) be the multigraph generated by the pairing model with degree sequence \((d_i^{(n)})_{i=1}^n\). Then for every \(\varepsilon>0\) there exists \(\eta=\eta(\Delta,\varepsilon)>0\) such that
\[
\Pbb\Bigl(
\forall S\subset[n]\text{ with }|S|\le \eta n:\ 
e_{\widetilde G_n}(S)\le (1+\varepsilon)|S|
\Bigr)\xrightarrow[n\to\infty]{}1.
\]
\end{lemma}

\begin{proof}
Write $M_n:=\sum_{i=1}^n d_i^{(n)}$.
Since \(d_i^{(n)}\ge 1\) for all \(i\), we have \(M_n\ge n\). Fix a set \(S\subset[n]\) with \(|S|=k\), and let $D_S:=\sum_{i\in S} d_i^{(n)}$.
By the bounded-degree assumption, $D_S\le \Delta k$. Set $m:=\lceil (1+\varepsilon)k\rceil$.
If \(e_{\widetilde G_n}(S)\ge (1+\varepsilon)k\), then necessarily
\[
e_{\widetilde G_n}(S)\ge m.
\]
We therefore estimate the probability that there are at least \(m\) edges of \(\widetilde G_n\) with both endpoints in \(S\).
In the pairing model, the total number of pairings of the \(M_n\) half-edges is $(M_n-1)!!$.
To obtain a pairing with at least \(m\) internal edges inside \(S\), it is enough to choose \(2m\) half-edges among the \(D_S\) half-edges incident to \(S\), pair them internally, and then pair the remaining \(M_n-2m\) half-edges arbitrarily. Hence
\[
\Pbb\bigl(e_{\widetilde G_n}(S)\ge m\bigr)
\le
\frac{\binom{D_S}{2m}(2m-1)!!(M_n-2m-1)!!}{(M_n-1)!!}.
\]
Using
\[
\frac{(M_n-2m-1)!!}{(M_n-1)!!}
=
\frac{1}{(M_n-1)(M_n-3)\cdots(M_n-2m+1)}
\le
\frac{1}{(M_n-2m)^m},
\]
and
\[
\binom{D_S}{2m}(2m-1)!!
=
\frac{D_S!}{(D_S-2m)!}\frac{1}{2^m m!}
\le
\frac{D_S^{2m}}{2^m m!},
\]
we obtain
\[
\Pbb\bigl(e_{\widetilde G_n}(S)\ge m\bigr)
\le
\frac{1}{m!}
\left(\frac{D_S^2}{2(M_n-2m)}\right)^m.
\]
Now \(m!\ge (m/e)^m\), so
\[
\Pbb\bigl(e_{\widetilde G_n}(S)\ge m\bigr)
\le
\left(
\frac{eD_S^2}{2m(M_n-2m)}
\right)^m.
\]

We now restrict to \(k\le \eta n\), where \(\eta>0\) will be chosen small. Since
\[
m=\lceil(1+\varepsilon)k\rceil\le (2+\varepsilon)k,
\]
if \(\eta\) is so small that
\[
2(2+\varepsilon)\eta\le \frac12,
\]
then for all \(k\le \eta n\),
\[
M_n-2m\ge n-2(2+\varepsilon)k\ge \frac n2.
\]
Using also \(D_S\le \Delta k\) and \(m\ge (1+\varepsilon)k\), we get
\[
\Pbb\bigl(e_{\widetilde G_n}(S)\ge (1+\varepsilon)k\bigr)
\le
\Pbb\bigl(e_{\widetilde G_n}(S)\ge m\bigr)
\le
\left(
\frac{e\Delta^2}{1+\varepsilon}\frac{k}{n}
\right)^m
\le
\left(
C_{\Delta,\varepsilon}\frac{k}{n}
\right)^{(1+\varepsilon)k}
\]
for a finite constant \(C_{\Delta,\varepsilon}>0\).

Taking a union bound over all \(S\subset[n]\) with \(|S|=k\), and using
\[
\binom{n}{k}\le \left(\frac{en}{k}\right)^k,
\]
we obtain
\[
\Pbb\Bigl(
\exists S\subset[n],\ |S|=k:\ e_{\widetilde G_n}(S)\ge (1+\varepsilon)k
\Bigr)
\le
\left(
\widetilde C_{\Delta,\varepsilon}\Bigl(\frac{k}{n}\Bigr)^\varepsilon
\right)^k
\]
for another finite constant \(\widetilde C_{\Delta,\varepsilon}>0\).

Now split the sum over \(k\) into the ranges \(1\le k\le \log n\) and \(\log n\le k\le \eta n\).

For \(1\le k\le \log n\),
\[
\left(
\widetilde C_{\Delta,\varepsilon}\Bigl(\frac{k}{n}\Bigr)^\varepsilon
\right)^k
\le
\left(
\widetilde C_{\Delta,\varepsilon}\Bigl(\frac{\log n}{n}\Bigr)^\varepsilon
\right)^k,
\]
hence
\[
\sum_{k=1}^{\lfloor\log n\rfloor}
\Pbb\Bigl(
\exists S,\ |S|=k:\ e_{\widetilde G_n}(S)\ge (1+\varepsilon)k
\Bigr)\xrightarrow[n\to\infty]{}0.
\]
Choose \(\eta>0\) even smaller, if necessary, so that
\[
\widetilde C_{\Delta,\varepsilon}\eta^\varepsilon<e^{-2}.
\]
Then for \(\log n\le k\le \eta n\),
\[
\left(
\widetilde C_{\Delta,\varepsilon}\Bigl(\frac{k}{n}\Bigr)^\varepsilon
\right)^k
\le
\left(\widetilde C_{\Delta,\varepsilon}\eta^\varepsilon\right)^k
\le e^{-2k},
\]
and therefore
\[
\sum_{k=\lceil\log n\rceil}^{\lfloor\eta n\rfloor}
\Pbb\Bigl(
\exists S,\ |S|=k:\ e_{\widetilde G_n}(S)\ge (1+\varepsilon)k
\Bigr)
\le
\sum_{k=\lceil\log n\rceil}^{\infty} e^{-2k}
\xrightarrow[n\to\infty]{}0.
\]

Combining the two ranges proves the claim.
\end{proof}

\begin{corollary}
\label{cor:A1-CM}
Assume that
\[
1\le d_i^{(n)}\le \Delta
\qquad\text{for all }i,n,
\]
and let \(G_n\) be the simple graph obtained from the configuration model with degree sequence \((d_i^{(n)})_{i=1}^n\), that is, the pairing model conditioned on simplicity. Then for every \(\varepsilon>0\) there exists \(\eta=\eta(\Delta,\varepsilon)>0\) such that
\[
\Pbb\Bigl(
\forall S\subset[n]\text{ with }|S|\le \eta n:\ 
e_{G_n}(S)\le (1+\varepsilon)|S|
\Bigr)\xrightarrow[n\to\infty]{}1.
\]
\end{corollary}

\begin{proof}
Let \(\mathcal E_n(\eta,\varepsilon)\) denote the event from the statement. By Lemma~\ref{lem:A1-CM-pairing}, for every \(\varepsilon>0\) there exists \(\eta>0\) such that
\[
\Pbb\bigl(\mathcal E_n(\eta,\varepsilon)\bigr)\to 1
\]
in the pairing model.

Moreover, for bounded degree sequences, the probability that the pairing model is simple is bounded away from \(0\) uniformly in \(n\); see for instance \cite[Theorem 1.1]{Janson2009}. Therefore
\[
\Pbb\bigl(\mathcal E_n(\eta,\varepsilon)^c \mid \widetilde G_n \text{ simple}\bigr)
\le
\frac{\Pbb(\mathcal E_n(\eta,\varepsilon)^c)}{\Pbb(\widetilde G_n \text{ simple})}
\xrightarrow[n\to\infty]{}0.
\]
Since conditioning on simplicity gives exactly the law of \(G_n\), the claim follows.
\end{proof}

\subsection{Random regular graphs}

\begin{corollary}
\label{cor:A1-RRG}
Fix \(d\ge 3\), and let \(G_n\) be the random \(d\)-regular graph on \(n\) vertices. Then for every \(\varepsilon>0\) there exists \(\eta=\eta(d,\varepsilon)>0\) such that
\[
p_{d,n}\bigl(\mathcal E_n(\eta,\varepsilon)\bigr)\xrightarrow[n\to\infty]{}1.
\]
\end{corollary}

\begin{proof}
Apply Corollary~\ref{cor:A1-CM} to the constant degree sequence \(d_i^{(n)}\equiv d\).
\end{proof}

\section{Proofs of the main theorems}

\subsection{An averaged lower bound}
The following proposition converts the deterministic drift bound into an averaged lower bound on the probability of a macroscopic loop under small-set sparsity and edge-density control. This abstract criterion will later allow us to deduce the averaged macroscopic-loop bounds for the random graph ensembles considered in this paper by verifying these two inputs with high probability.
\begin{proposition}
\label{prop:abstract-averaged}
Fix \(u\in[0,1]\) and \(\theta>0\). Define
\[
C_\theta:=
\begin{cases}
1, & \theta=1,\\[1ex]
\dfrac{\theta |\log \theta|}{|\theta-1|}, & \theta\neq 1.
\end{cases}
\]

Let \(G=(V,E)\) be a finite graph with \(|V|=n\), and assume that for some \(\eta,\varepsilon>0\),
\[
e_G(S)\le (1+\varepsilon)|S|
\qquad\text{for every }S\subset V\text{ with }|S|\le \eta n.
\]
Assume moreover that
\[
\frac{|E|}{n}\in [m_-,m_+]
\]
for some \(m_-,m_+>0\).
Then for every \(a\ge 0\) and every \(s>0\),
\[
\frac1s\int_a^{a+s}\Pbb^G_{\theta,t,u}(A_\eta)\,\dd t
\ge
\frac{
m_- - c_{\theta,u}(1+\varepsilon)-C_\theta/s
}{
(1+\theta)m_+ - c_{\theta,u}(1+\varepsilon)
},
\]
provided the denominator is positive.
\end{proposition}

\begin{proof}
By Proposition~\ref{prop:deterministic-drift},
\[
D_{\theta,u}(t)
\le
\bigl(c_{\theta,u}(1+\varepsilon)n-|E|\bigr)
+
\bigl((1+\theta)|E|-c_{\theta,u}(1+\varepsilon)n\bigr)\,
\Pbb^G_{\theta,t,u}(A_\eta).
\]
Rearranging, we obtain
\[
\Pbb^G_{\theta,t,u}(A_\eta)
\ge
\frac{
D_{\theta,u}(t)-c_{\theta,u}(1+\varepsilon)n+|E|
}{
(1+\theta)|E|-c_{\theta,u}(1+\varepsilon)n
}.
\]
Using \(|E|/n\le m_+\) in the denominator and \(|E|/n\ge m_-\) in the numerator gives
\[
\Pbb^G_{\theta,t,u}(A_\eta)
\ge
\frac{
D_{\theta,u}(t)-c_{\theta,u}(1+\varepsilon)n+m_- n
}{
(1+\theta)m_+ n-c_{\theta,u}(1+\varepsilon)n
}.
\]
Integrating over \([a,a+s]\), we get
\[
\frac1s\int_a^{a+s}\Pbb^G_{\theta,t,u}(A_\eta)\,\dd t
\ge
\frac{
\frac1s\int_a^{a+s}D_{\theta,u}(t)\,\dd t
-c_{\theta,u}(1+\varepsilon)n+m_- n
}{
(1+\theta)m_+ n-c_{\theta,u}(1+\varepsilon)n
}.
\]

It remains to bound the averaged drift term from below.
If \(\theta=1\), then
\[
D_{1,u}(t)=\frac{\dd}{\dd t}\Ebb^G_{1,t,u}[\lambda(\omega)].
\]
Since always \(1\le \lambda(\omega)\le n\), we have
\[
\frac1s\int_a^{a+s}D_{1,u}(t)\,\dd t
=
\frac{\Ebb[\lambda(\omega_{a+s})]-\Ebb[\lambda(\omega_a)]}{s}
\ge -\frac{n}{s}.
\]

If \(\theta\neq 1\), then
\[
D_{\theta,u}(t)=\frac{\theta}{\theta-1}\frac{\dd}{\dd t}\log Z_G(\theta,t,u).
\]
Moreover, since \(1\le \lambda(\omega)\le n\),
\[
\theta^n\le Z_G(\theta,t,u)\le \theta
\qquad\text{if }0<\theta<1,
\]
and
\[
\theta\le Z_G(\theta,t,u)\le \theta^n
\qquad\text{if }\theta>1.
\]
Hence in both cases,
\[
\left|
\log Z_G(\theta,a+s,u)-\log Z_G(\theta,a,u)
\right|
\le n|\log\theta|.
\]
Therefore
\[
\frac1s\int_a^{a+s}D_{\theta,u}(t)\,\dd t
=
\frac{\theta}{\theta-1}\,
\frac{\log Z_G(\theta,a+s,u)-\log Z_G(\theta,a,u)}{s}
\ge
-\frac{\theta |\log\theta|}{|\theta-1|}\,\frac{n}{s}.
\]

Thus, in all cases,
\[
\frac1s\int_a^{a+s}D_{\theta,u}(t)\,\dd t
\ge - C_\theta \frac{n}{s}.
\]
Substituting this into the previous bound yields
\[
\frac1s\int_a^{a+s}\Pbb^G_{\theta,t,u}(A_\eta)\,\dd t
\ge
\frac{
m_- - c_{\theta,u}(1+\varepsilon)-C_\theta/s
}{
(1+\theta)m_+ - c_{\theta,u}(1+\varepsilon)
},
\]
as claimed.
\end{proof}

\subsection{Proof of the averaged criterion}

\begin{proof}[Proof of Theorem~\ref{thm:abstract-averaged}]
Choose \(\varepsilon>0\) so small that
\[
\alpha>c_{\theta,u}(1+\varepsilon).
\]
Then choose \(\xi>0\) so small that $\alpha-\xi>c_{\theta,u}(1+\varepsilon)$.
By assumption \textnormal{(a)}, there exists \(\eta>0\) such that
\[
\Pbb\bigl(\mathcal E_n(\eta,\varepsilon)\bigr)\xrightarrow[n\to\infty]{}1.
\]
By assumption \textnormal{(b)},
\[
\Pbb\!\left(
\frac{|E(G_n)|}{n}\in[\alpha-\xi,\alpha+\xi]
\right)\xrightarrow[n\to\infty]{}1.
\]
Hence, with probability tending to \(1\), both events occur simultaneously. On this intersection,
Proposition~\ref{prop:abstract-averaged} applies with
\[
m_-=\alpha-\xi,
\qquad
m_+=\alpha+\xi.
\]
Therefore, for every \(a\ge 0\) and every \(s>0\),
\[
\frac1s\int_a^{a+s}\Pbb^{G_n}_{\theta,t,u}(A_\eta)\,\dd t
\ge
\frac{
\alpha-\xi-c_{\theta,u}(1+\varepsilon)-C_\theta/s
}{
(1+\theta)(\alpha+\xi)-c_{\theta,u}(1+\varepsilon)
}.
\]
Choose \(s>0\) so large that
\[
\alpha-\xi-c_{\theta,u}(1+\varepsilon)-C_\theta/s>0.
\]
Then the right-hand side is a strictly positive constant \(c>0\). Since the two required events
hold with probability tending to \(1\), we conclude that
\[
\Pbb\!\left(
\frac1s\int_a^{a+s}\Pbb^{G_n}_{\theta,t,u}(A_\eta)\,\dd t \ge c
\right)\xrightarrow[n\to\infty]{}1.
\]
This proves the theorem.
\end{proof}

\begin{proof}[Proof of Corollary~\ref{cor:abstract-good-time}]
If
\[
\frac1s\int_a^{a+s}\Pbb^{G_n}_{\theta,t,u}(A_\eta)\,\dd t\ge c,
\]
then there exists \(t\in[a,a+s]\) such that
\[
\Pbb^{G_n}_{\theta,t,u}(A_\eta)\ge c.
\]
Therefore the conclusion follows immediately from
Theorem~\ref{thm:abstract-averaged}.
\end{proof}

\subsection{Verification of the hypotheses for the examples}

\begin{lemma}[Edge density for sparse Erd\H{o}s--R\'enyi graphs]
\label{lem:ER-edge-density}
Let \(G_n\sim G(n,\lambda/n)\). Then
\[
\frac{|E(G_n)|}{n}\xrightarrow[n\to\infty]{\Pbb}\frac{\lambda}{2}.
\]
\end{lemma}

\begin{proof}
Since
\[
|E(G_n)|\sim \mathrm{Bin}\!\Bigl(\binom{n}{2},\frac{\lambda}{n}\Bigr),
\]
we have
\[
\Ebb[|E(G_n)|]=\binom{n}{2}\frac{\lambda}{n}
=\frac{\lambda}{2}(n-1),
\]
and
\[
\mathrm{Var}(|E(G_n)|)
=
\binom{n}{2}\frac{\lambda}{n}\Bigl(1-\frac{\lambda}{n}\Bigr)
=O(n).
\]
Hence
\[
\mathrm{Var}\!\left(\frac{|E(G_n)|}{n}\right)=O\!\left(\frac1n\right)\to 0,
\]
and the claim follows from Chebyshev's inequality.
\end{proof}

\begin{proof}[Proof of Corollary~\ref{cor:examples-averaged}]
We verify the assumptions of Theorem~\ref{thm:abstract-averaged} in each case.

\smallskip
\noindent\textit{Proof of part \textnormal{(i)}.}
Let \(G_n\) be the random \(d\)-regular graph on \([n]\). By
Corollary~\ref{cor:A1-RRG}, for every \(\varepsilon>0\) there exists \(\eta>0\) such that
\[
p_{d,n}\bigl(\mathcal E_n(\eta,\varepsilon)\bigr)\xrightarrow[n\to\infty]{}1.
\]
Moreover,
\[
|E(G_n)|=\frac{dn}{2}
\]
deterministically, so
\[
\frac{|E(G_n)|}{n}\to \frac d2.
\]
Since \(d>2c_{\theta,u}\), we have \(d/2>c_{\theta,u}\). Hence
Theorem~\ref{thm:abstract-averaged} applies with \(\alpha=d/2\).

\smallskip
\noindent\textit{Proof of part \textnormal{(ii)}.}
Let \(G_n\sim G(n,\lambda/n)\). By Lemma~\ref{lem:A1-ER}, for every \(\varepsilon>0\) there
exists \(\eta>0\) such that
\[
\Pbb\bigl(\mathcal E_n(\eta,\varepsilon)\bigr)\xrightarrow[n\to\infty]{}1.
\]
By Lemma~\ref{lem:ER-edge-density},
\[
\frac{|E(G_n)|}{n}\xrightarrow[n\to\infty]{\Pbb}\frac{\lambda}{2}.
\]
Since \(\lambda>2c_{\theta,u}\), we have \(\lambda/2>c_{\theta,u}\). Hence
Theorem~\ref{thm:abstract-averaged} applies with \(\alpha=\lambda/2\).

\smallskip
\noindent\textit{Proof of part \textnormal{(iii)}.}
Let \(G_n\) be the simple bounded-degree configuration model associated with the degree sequence
\((d_i^{(n)})_{i=1}^n\). By Corollary~\ref{cor:A1-CM}, for every \(\varepsilon>0\) there exists
\(\eta>0\) such that
\[
\Pbb\bigl(\mathcal E_n(\eta,\varepsilon)\bigr)\xrightarrow[n\to\infty]{}1.
\]
Since the degree sequence is deterministic and conditioning on simplicity does not change the
degree sequence,
\[
|E(G_n)|=\frac12\sum_{i=1}^n d_i^{(n)},
\]
and therefore
\[
\frac{|E(G_n)|}{n}
=
\frac12\cdot \frac1n\sum_{i=1}^n d_i^{(n)}
\longrightarrow \frac{\rho}{2}
\]
deterministically. Since \(\rho>2c_{\theta,u}\), we have \(\rho/2>c_{\theta,u}\). Hence
Theorem~\ref{thm:abstract-averaged} applies with \(\alpha=\rho/2\).
\end{proof}

\subsection{A general pointwise lower bound}
The following proposition converts the deterministic drift bound into a pointwise lower bound under small-set sparsity, edge-density control, and log-convexity. It will later be applied to the concrete random graph ensembles considered in this paper.
\begin{proposition}[Pointwise lower bound under sparsity, edge-density control, and log-convexity]
\label{prop:abstract-pointwise}
Fix \(u\in[0,1]\) and \(\theta\in\mathbb N\) with \(\theta>1\). Let \(G=(V,E)\) be a finite graph
with \(|V|=n\), and assume that for some \(\eta,\varepsilon>0\),
\[
e_G(S)\le (1+\varepsilon)|S|
\qquad\text{for every }S\subset V\text{ with }|S|\le \eta n.
\]
Assume moreover that
\[
\frac{|E|}{n}\in [m_-,m_+]
\]
for some \(m_-,m_+>0\), and that
\[
m_->c_{\theta,u}(1+\varepsilon).
\]
Then, for every
\[
t>
\frac{\theta\log\theta}{(\theta-1)\bigl(m_- - c_{\theta,u}(1+\varepsilon)\bigr)},
\]
one has
\[
\Pbb^G_{\theta,t,u}(A_\eta)
\ge
\frac{
m_- - c_{\theta,u}(1+\varepsilon)-\frac{\theta\log\theta}{(\theta-1)t}
}{
(1+\theta)m_+ - c_{\theta,u}(1+\varepsilon)
}.
\]
\end{proposition}

\begin{proof}
By Proposition~\ref{prop:deterministic-drift},
\[
D_{\theta,u}(t)
\le
\bigl(c_{\theta,u}(1+\varepsilon)n-|E|\bigr)
+
\bigl((1+\theta)|E|-c_{\theta,u}(1+\varepsilon)n\bigr)\,
\Pbb^G_{\theta,t,u}(A_\eta).
\]
Rearranging,
\[
\Pbb^G_{\theta,t,u}(A_\eta)
\ge
\frac{
D_{\theta,u}(t)-c_{\theta,u}(1+\varepsilon)n+|E|
}{
(1+\theta)|E|-c_{\theta,u}(1+\varepsilon)n
}.
\]
Using \(|E|/n\ge m_-\) in the numerator and \(|E|/n\le m_+\) in the denominator yields
\[
\Pbb^G_{\theta,t,u}(A_\eta)
\ge
\frac{
D_{\theta,u}(t)-c_{\theta,u}(1+\varepsilon)n+m_-n
}{
(1+\theta)m_+n-c_{\theta,u}(1+\varepsilon)n
}.
\]

It remains to bound \(D_{\theta,u}(t)\) from below. By Proposition~\ref{prop:logconvex}, the function
\[
t\longmapsto \log Z_G(\theta,t,u)
\]
is convex. Therefore
\[
\partial_t\log Z_G(\theta,t,u)
\ge
\frac{\log Z_G(\theta,t,u)-\log Z_G(\theta,0,u)}{t}.
\]
Since \(1\le \lambda(\omega)\le n\), we have
\[
\theta\le Z_G(\theta,t,u)\le \theta^n,
\qquad
Z_G(\theta,0,u)=\theta^n,
\]
and hence
\[
\partial_t\log Z_G(\theta,t,u)\ge -\frac{n\log\theta}{t}.
\]
Therefore
\[
D_{\theta,u}(t)
=
\frac{\theta}{\theta-1}\partial_t\log Z_G(\theta,t,u)
\ge
-\frac{\theta\log\theta}{\theta-1}\frac{n}{t}.
\]
Substituting this lower bound into the previous estimate gives
\[
\Pbb^G_{\theta,t,u}(A_\eta)
\ge
\frac{
m_- - c_{\theta,u}(1+\varepsilon)-\frac{\theta\log\theta}{(\theta-1)t}
}{
(1+\theta)m_+ - c_{\theta,u}(1+\varepsilon)
}.
\]
This is positive precisely under the stated condition on \(t\).
\end{proof}

\subsection{Proof of the pointwise criterion}

\begin{proof}[Proof of Theorem~\ref{thm:abstract-pointwise}]
Choose \(\varepsilon>0\) so small that
\[
\alpha>c_{\theta,u}(1+\varepsilon).
\]
Then choose \(\xi>0\) so small that
\[
\alpha-\xi>c_{\theta,u}(1+\varepsilon)
\]
and
\[
\frac{\theta\log\theta}{(\theta-1)\bigl(\alpha-\xi-c_{\theta,u}(1+\varepsilon)\bigr)}
<
\frac{\theta\log\theta}{(\theta-1)(\alpha-c_{\theta,u})}+\frac{\delta}{2}.
\]
By assumption \textnormal{(a)}, there exists \(\eta>0\) such that
\[
\Pbb\bigl(\mathcal E_n(\eta,\varepsilon)\bigr)\xrightarrow[n\to\infty]{}1.
\]
By assumption \textnormal{(b)},
\[
\Pbb\!\left(
\frac{|E(G_n)|}{n}\in[\alpha-\xi,\alpha+\xi]
\right)\xrightarrow[n\to\infty]{}1.
\]
Hence, with probability tending to \(1\), both events occur simultaneously. On their intersection,
Proposition~\ref{prop:abstract-pointwise} applies with
\[
m_-=\alpha-\xi,
\qquad
m_+=\alpha+\xi.
\]
Therefore, for every
\[
t\ge \frac{\theta\log\theta}{(\theta-1)(\alpha-c_{\theta,u})}+\delta,
\]
we have
\[
\Pbb^{G_n}_{\theta,t,u}(A_\eta)
\ge
\frac{
\alpha-\xi-c_{\theta,u}(1+\varepsilon)-\frac{\theta\log\theta}{(\theta-1)t}
}{
(1+\theta)(\alpha+\xi)-c_{\theta,u}(1+\varepsilon)
}.
\]
By the choice of \(\xi\) and \(\varepsilon\), the numerator is strictly positive at
\[
t=\frac{\theta\log\theta}{(\theta-1)(\alpha-c_{\theta,u})}+\delta,
\]
and hence bounded below by a positive constant \(c>0\) uniformly for all larger \(t\). Since the
two required events hold with probability tending to \(1\), the conclusion follows.
\end{proof}

\begin{proof}[Proof of Corollary~\ref{cor:examples-pointwise}]
We verify the assumptions of Theorem~\ref{thm:abstract-pointwise} in each case.

\smallskip
\noindent\textit{Proof of part \textnormal{(i)}.}
Let \(G_n\) be the random \(d\)-regular graph on \([n]\), with law \(p_{d,n}\). As in the proof of
Corollary~\ref{cor:examples-averaged}, Corollary~\ref{cor:A1-RRG} gives assumption \textnormal{(a)},
while
\[
\frac{|E(G_n)|}{n}\to \frac d2
\]
holds deterministically. Hence Theorem~\ref{thm:abstract-pointwise} applies with \(\alpha=d/2\).
Since
\[
\frac{\theta\log\theta}{(\theta-1)(d/2-c_{\theta,u})}
=
T_{\theta,u}(d),
\]
this gives the stated conclusion.

\smallskip
\noindent\textit{Proof of part \textnormal{(ii)}.}
Let \(G_n\sim G(n,\lambda/n)\). By Lemma~\ref{lem:A1-ER}, assumption \textnormal{(a)} holds, and
by Lemma~\ref{lem:ER-edge-density},
\[
\frac{|E(G_n)|}{n}\xrightarrow[n\to\infty]{\Pbb}\frac{\lambda}{2}.
\]
Hence Theorem~\ref{thm:abstract-pointwise} applies with \(\alpha=\lambda/2\). Since
\[
\frac{\theta\log\theta}{(\theta-1)(\lambda/2-c_{\theta,u})}
=
T_{\theta,u}(\lambda),
\]
the desired statement follows.

\smallskip
\noindent\textit{Proof of part \textnormal{(iii)}.}
Let \(G_n\) be the simple bounded-degree configuration model associated with the degree sequence
\((d_i^{(n)})_{i=1}^n\). By Corollary~\ref{cor:A1-CM}, assumption \textnormal{(a)} holds, and
\[
\frac{|E(G_n)|}{n}\to \frac{\rho}{2}
\]
deterministically, as shown above. Hence Theorem~\ref{thm:abstract-pointwise} applies with
\(\alpha=\rho/2\). Since
\[
\frac{\theta\log\theta}{(\theta-1)(\rho/2-c_{\theta,u})}
=
T_{\theta,u}(\rho),
\]
this yields the stated conclusion.
\end{proof}

\appendix

\section{Log-convexity of the partition function for integer loop weights}
\label{app:log-convexity}
For completeness, we prove in this appendix the following log-convexity property of the partition function.

\begin{proposition}[Log-convexity of the partition function for integer loop weights]
\label{prop:logconvex}
Assume \(\theta\in\mathbb N\) and \(\theta>1\). Then, for every finite graph
\(G=(V,E)\) and every \(u\in[0,1]\), the map
\[
t\longmapsto Z_G(\theta,t,u)
\]
is log-convex on \([0,\infty)\), that is,
\[
\frac{\dd^2}{\dd t^2}\log Z_G(\theta,t,u)\ge 0
\qquad\text{for all } t\ge 0.
\]
\end{proposition}

We first record the standard trace representation.

\begin{proposition}[Standard trace representation for integer loop weights]
\label{prop:trace-representation}
Let \(\theta\in\mathbb N\), \(u\in[0,1]\), and let \(G=(V,E)\) be a finite graph. Then there exists a finite-dimensional self-adjoint operator \(H_{G,\theta,u}\) such that
\[
Z_G(\theta,t,u)=\operatorname{Tr}\bigl(e^{-tH_{G,\theta,u}}\bigr)
\qquad\text{for all }t\ge 0.
\]
\end{proposition}

\begin{proof}
This is the standard random loop representation for quantum spin systems with spin $S=\frac{\theta-1}{2}$,
for which the loop weight is \(2S+1=\theta\); see \cite[Section~3]{Ueltschi2013}. In that representation, the partition function can be written in the form
\[
Z_G(\theta,t,u)=\operatorname{Tr}\bigl(e^{-tH_{G,\theta,u}}\bigr)
\]
for a finite-dimensional self-adjoint Hamiltonian \(H_{G,\theta,u}\).
\end{proof}

\begin{proof}[Proof of Proposition~\ref{prop:logconvex}]
By Proposition~\ref{prop:trace-representation}, there exists a finite-dimensional self-adjoint operator \(H_{G,\theta,u}\) such that
\[
Z_G(\theta,t,u)=\operatorname{Tr}\bigl(e^{-tH_{G,\theta,u}}\bigr).
\]
Let
\[
\lambda_1,\dots,\lambda_M
\]
denote the eigenvalues of \(H_{G,\theta,u}\), counted with multiplicity. Then
\[
Z_G(\theta,t,u)=\sum_{j=1}^M e^{-t\lambda_j}.
\]
Differentiating, we get
\[
Z_G'(\theta,t,u)=-\sum_{j=1}^M \lambda_j e^{-t\lambda_j},
\qquad
Z_G''(\theta,t,u)=\sum_{j=1}^M \lambda_j^2 e^{-t\lambda_j}.
\]
Therefore
\[
\frac{\dd^2}{\dd t^2}\log Z_G(\theta,t,u)
=
\frac{Z_G''(\theta,t,u)}{Z_G(\theta,t,u)}
-
\left(\frac{Z_G'(\theta,t,u)}{Z_G(\theta,t,u)}\right)^2.
\]
Substituting the above expressions yields
\[
\frac{\dd^2}{\dd t^2}\log Z_G(\theta,t,u)
=
\frac{\sum_{j=1}^M \lambda_j^2 e^{-t\lambda_j}}{\sum_{j=1}^M e^{-t\lambda_j}}
-
\left(
\frac{\sum_{j=1}^M \lambda_j e^{-t\lambda_j}}{\sum_{j=1}^M e^{-t\lambda_j}}
\right)^2.
\]
The right-hand side is the variance of the numbers \(\lambda_1,\dots,\lambda_M\) with respect to the probability weights
\[
p_j(t):=\frac{e^{-t\lambda_j}}{\sum_{k=1}^M e^{-t\lambda_k}},
\]
and is therefore nonnegative. Hence
\[
\frac{\dd^2}{\dd t^2}\log Z_G(\theta,t,u)\ge 0
\qquad\text{for all }t\ge 0.
\]
This proves the claim.
\end{proof}

\begin{remark}
The trace representation used above is standard in the random loop representation of quantum spin systems; in the notation of \cite{Ueltschi2013}, the loop weight is \(2S+1\), which yields all integer values of \(\theta\).
\end{remark}

\noindent\textbf{\large Acknowledgement}. The author is grateful to Daniel Ueltschi for drawing his attention to the paper of Poudevigne--Auboiron during a stay in Warwick in February 2025.

\medskip

\noindent\textbf{Funding acknowledgement.} The author gratefully acknowledges financial support from Bischöfliche Studienförderung Cusanuswerk.

\bibliographystyle{plain}
\bibliography{random_loops_refs}

\end{document}